 \documentclass[titlepage]{amsart}
\usepackage{amsaddr}
\newcommand\blfootnote[1]{%
	\begingroup
	\renewcommand\thefootnote{}\footnote{#1}%
	\addtocounter{footnote}{-1}%
	\endgroup
}

\usepackage[cm]{fullpage}
\usepackage{amssymb,amsmath,amsthm,amscd,mathrsfs,graphicx}
\usepackage[cmtip,all]{xy}
\usepackage{tikz}

\numberwithin{equation}{section}
\newtheorem{teo}{Theorem}[section]

\newtheorem{theorem}{Theorem}[section]
\newtheorem{proposition}{Proposition}[section]

\newtheorem{corollary}{Corollary}[section]
\newtheorem{conjecture}{Conjecture}[section]

\theoremstyle{definition}

\theoremstyle{remark}
\newtheorem{remark}[teo]{Remark}

\begin{document}
	
	\author{Albert Chau$^*$}
	\address{Department of Mathematics,
		The University of British Columbia, 1984 Mathematics
		Road, Vancouver, B.C., Canada V6T 1Z2} \email{chau@math.ubc.ca}

	\thanks{$^*$Research
		partially supported by NSERC grant no. \#327637-06}

	\author[A. Martens]{Adam Martens}
	\address{Department of Mathematics, The University of British Columbia, 1984 Mathematics Road, Vancouver, B.C.,  Canada V6T 1Z2}  \email{martens@math.ubc.ca}

\blfootnote{MSC class: 53E20}
\blfootnote{Statements and Declaration: the first author was partially supported by NSERC grant no. \#327637-06.}

	\bibliographystyle{amsplain}

\title{Pseudolocality and completeness for nonnegative Ricci curvature limits of 3D singular Ricci flows}
%	\title{On completeness of 3D noncompact Ricci flows with nonnegative Ricci curvature}

\begin{abstract} Lai (2021) used singular Ricci flows, introduced by Kleiner and Lott (2017), to construct a nonnegative Ricci curvature Ricci flow $g(t)$ emerging from an arbitrary 3D complete noncompact Riemannian manifold $(M^3, g_0)$ with nonnegative Ricci curvature.  We show $g(t)$ is complete for positive times provided $g_0$ satisfies a volume ratio lower bound that approaches zero at spatial infinity.  Our proof combines a pseudolocality  result of Lai (2021) for singular flows, together with a pseudolocality result of Hochard (2016) and Simon and Topping (2022) for nonsingular flows.  We also show that the construction of complete nonnegative complex sectional curvature flows by Cabezas-Rivas and Wilking (2015) can be adapted here to show $g(t)$ is complete for positive times provided $g_0$ is a compactly supported perturbation of a nonnegative sectional curvature metric.

   \end{abstract}

	\maketitle
	
	\section{Introduction}
	
 In the seminal work \cite{rH1}, Hamilton introduced the Ricci flow  which is the following evolution equation for a family of Riemannian metrics $g(t)$ starting from an initial smooth $n$ dimensional Riemannian manifold $(M^n, g_0)$:

 \begin{equation}\label{RF}
  \left\{
 \begin{aligned}
 	\frac{{\partial g}}{{\partial t}} &= -2 \mathrm{Ric}(g), \\
 	g(0) &= g_0.
 \end{aligned}
 \right.
 \end{equation}
It was proved in \cite{rH1} that when $M^n$ is compact, the Ricci flow \eqref{RF} admits a unique smooth solution $g(t)$ on $M^n\times [0, T)$ for a maximal time $T>0$ bounded below depending only on the dimension $n$ and the initial bounds on the sectional curvatures of $g_0$.  Moreover, when $g_0$ has nonnegative Ricci curvature and $n=3$, Hamilton's collective results in \cite{rH1}, \cite{rH2} and \cite{rH3} imply that $g(t)$ has nonnegative Ricci curvature for all $t>0$ and is either Ricci flat for all times or else converges, after appropriate scaling and pulling back to the universal cover, to the standard metric on either $	\mathbb{S}^3$ or $	\mathbb{S}^2\times 	\mathbb{R}$ as $t\to T$.

\
   It is natural to wonder about the extent to which similar results hold when $M^n$ is noncompact.   This was initiated by Shi in \cite{Shi} who proved the existence of a complete bounded curvature solution $g(t)$ in any dimension assuming $g_0$ is also complete with bounded curvature.  Shi also showed \cite{Shi2} that when $g_0$ has nonnegative Ricci curvature and $n=3$, the solution actually converges, after appropriate scaling and pulling back to the universal cover, towards the standard metric on $	\mathbb{R}^3$ or else $	\mathbb{S}^2 \times 	\mathbb{R}$.  When $g_0$ is complete with possibly unbounded curvature, one cannot expect a solution to \eqref{RF} in general.  In fact, given any $\alpha>0$, it is expected that there exists a complete metric $g_0$ on $	\mathbb{S}^2 \times 	\mathbb{R}$ with $\operatorname{Ric}(g_0)\geq -\alpha$ which exhibits no complete solution to \eqref{RF} (see example 4 in \cite{T1}).  On the other hand, the following conjecture has been considered for a long time and is a special case of a conjecture by Topping for dimensions $n\geq 3$ {\cite[Conjecture 1.1]{T0}}.

    \begin{conjecture}\label{mainquestion}
   Let $(M^3, g_0)$ be a complete noncompact $3$-dimensional Riemannian manifold with nonnegative (possibly unbounded)   Ricci curvature $\operatorname{Ric}(g_0) \geq 0$.  Then \eqref{RF} has a corresponding smooth solution $g(t)$ on $M^3 \times[0, T)$ for some $T>0$, and $g(t)$ is complete and has nonnegative Ricci curvature for each $t\in [0, T)$.    	 
    \end{conjecture}

   	The theory of \eqref{RF} for complete unbounded curvature metrics $g_0$ has seen significant developments since the above mentioned works.  Those developments which support Conjecture \ref{mainquestion} include the following:  Cabezas-Rivas and Wilking \cite{CRW} showed Conjecture \ref{mainquestion} holds when ``nonnegative Ricci curvature" is replaced by ``nonnegative sectional curvature".  Hochard \cite{raH} showed that if $g_0$ is complete with Ricci curvature bounded below (not necessarily by zero), and there exists a uniform positive lower bound on the volume of initial unit balls, then \eqref{RF} has a short-time complete solution $g(t)$.  On the other hand, Chen $et$ $al.$ \cite{CXZ} showed that the hypothesis of nonnegative Ricci curvature is preserved along any 3D complete solution to \eqref{RF}.  Combining these shows that Conjecture \ref{mainquestion} holds under the assumption of a uniform lower bound on volume of initial unit balls. Lai \cite{yL} proved that Conjecture \ref{mainquestion} holds provided we remove the condition of completeness of the solution $g(t)$ for each $t\in (0, T)$.  In other words, there exists a nonnegative Ricci curvature, but possibly instantaneously incomplete, solution $g(t)$ to \eqref{RF} emerging from any complete nonnegative Ricci curvature metric $g_0$.  Lee and Topping \cite{LT} showed that Conjecture \ref{mainquestion} holds provided the pinching condition Ric$(g_0)\geq \epsilon R(g_0)\geq 0$ holds for some $\epsilon>0$ (where $R$ is scalar curvature), and that the solution $g(t)$ has bounded sectional curvatures and likewise pinched Ricci curvature for all $t>0$.  Combining this with earlier results of Deruelle $et$ $al.$ \cite{DSS} and Lott \cite{L}, they were able to conclude that $(M^3, g_0)$ must in fact be either compact or else flat thus proving Hamilton's pinching conjecture.

   %	We refer to \cite{....} for partial extensions of some of the above results to dimensions $n\geq 4$.  We also refer to \cite{...} for results in dimension $n=2$ relating to those on 3D singular flows in \cite{yL}. 

    Conjecture \ref{mainquestion} can thus be reduced to showing the completeness of the specific solution constructed in \cite{yL} which is the approach we adopt in this article.  The construction in \cite{yL} in turn, is based on studying the singular Ricci flows $\mathcal{N}_k$ emerging from each member of an arbitrary sequence of compact Riemannian manifolds without boundaries $\{(N_k, g_k)\}$ approximating $(M^3, g_0)$.  Singular 3-dimensional Ricci flows, introduced by Kleiner and Lott \cite{KL}, are 4-dimensional Ricci flow space-times emerging from a given compact 3-dimensional manifold $(M^3, g_0)$ and satisfying certain asymptotic conditions.  In many cases,  the ``classical" maximal solution $(M^3\times[0, T_{g_0}), g_0)$ will be strictly contained within the corresponding singular Ricci flow.  The reason for using singular flows is because the classical solutions to \eqref{RF} emerging from $g_k$ may not exist up to a time which is uniform in $k$, in which case one could not use these to obtain a limit solution on $M^3\times[0, T)$ for any $T>0$.  By combining the construction in \cite{yL} with the 3D pseudolocality of Simon and Topping \cite{ST} we prove

	  \begin{theorem}\label{t1}	
	  	There exists a function $f(R):	\mathbb{R}^+\to 	\mathbb{R}^+$ with $\lim_{R\to \infty} f(R) =0$ such that if $(M^3, g_0)$ is a complete 3-dimensional Riemannian manifold with nonnegative Ricci curvature and

	  \begin{equation}\label{t1c1}
	  	\operatorname{Vol}_{g_0}(B_{g_0} (p,R))\geq f(R) R^3
	  \end{equation}
	  	for some $p\in M^3$ and all $R$ sufficiently large, then the Ricci flow \eqref{RF} admits a smooth short-time solution $g(t)$ that starts from $g_0$, is complete and has nonnegative Ricci curvature for every $t > 0$.

	  \end{theorem}
Any complete 3-dimensional Riemannian manifold with nonnegative Ricci curvature has at least linear volume growth \cite{Y} in the sense that \eqref{t1c1} holds for  $f(R)=C_p/R^2$ for some $C_p$.  On the other hand, Euclidean volume growth corresponds to when \eqref{t1c1} holds for  $f(R)=C$ for some constant $C$, and in this case Bishop-Gromov volume comparison implies a uniform lower bound on volume of unit balls in which case our results follows from \cite{raH} and \cite{CXZ} as mentioned above.  In general, the function $f$ provides a lower bound on the volume of a unit ball at any $q\in M^3$ as
  $$V(q, 1)\geq \frac{V(q, 2 d(p, q))}{(2 d(p, q))^3}\geq \frac{V(p, d(p, q))}{8(d(p, q))^3}\geq \frac{1}{8}f(d(p, q))$$
  where we have used Bishop-Gromov volume comparison for the first inequality, the triangle inequality for the second, and \eqref{t1c1} for the last, and we have abbreviated $\operatorname{Vol}_{g_0}(B_{g_0} (x,r))$ with  $V(x, r)$.

   The proof of Theorem \ref{t1} is presented in \S3 though we provide the following outline.  Let $(M^3,g_0)$ be complete noncompact with nonnegative Ricci curvature $\operatorname{Ric}(g_0)\geq 0$, and assume that $M^3$ is orientable.
    Then we can find an exhaustion of $M^3$ by compact connected sets $V_k$ with smooth boundaries, corresponding smooth compact Riemannian manifolds without boundaries $(N_k, h_k)$ where $N_k$ is the topological double of $V_k$, and maps $\phi_k:(V_k, g_0)\to (N_k, h_k)$ which are isometries when restricted to $V_{k-1}$.  In particular, we obtain a sequence of compact Ricci flows $\{(N_k, h_k(t)), \; t\in [0,T_k)\}$ with $h_k(0)=h_k$.  It is possible here that $T_k\to 0$ as $k\to \infty$.   However, by analyzing the corresponding singular Ricci flows, Lai  {\cite[Theorem 7.14 and 8.4]{yL}} showed that for each $V_k$, there exists $l_k$ such that $h_{l_k}(t)$ can be extended to $\phi_{l_k}(V_k)\times[0, T)$ for a uniform $T>0$ and that the pullbacks $\{V_k\times[0, T), g_k(t):=\phi_{l_k}^*(h_{l_k}(t))\}$ smoothly locally converges to a nonnegative Ricci curvature (possibly incomplete) Ricci flow $(M^3,g(t))$, $t\in [0,T)$ with $g(0)=g$. Moreover, if $(M^3,g_0)$ is the oriented double cover of a nonorientable $(M',g')$, then $g(t)$ pushes down to a solution $g'(t)$ on $M'\times [0,T)$. The above logic is summarized in Theorem \ref{tyl}.
   
  It was also proved by Lai {\cite[Theorem 6.1]{yL}} that Perelman's original pseudolocality statement {\cite[Theorem 10.1]{P}} still holds for the incomplete flows $(V_k \times[0, T),g_k(t))$ (see Theorem \ref{pseudolocalityyilai0}). In Theorem \ref{pseudolocalityyilai}, we modify this to one analogous to Perelman's second pseudolocality theorem {\cite[Theorem 10.3]{P}}. Armed with this, we are finally able to prove completeness of the solution $g(t)$ constructed above as follows. The idea is to apply a pseudolocality result by Hochard {\cite[Theorem 2.4]{raH}} and Simon and Topping {\cite[Theorem 1.1]{ST} which, a priori applies only to complete bounded curvature flows, and concludes a rough curvature bound $\sup_{B_{g(0)}(x_0,1)}|\operatorname{Rm}|_{g(t)}\leq c(v)/t$ from weak volume control $\operatorname{Vol} \,B_{g(0)}(x_0,1)\geq v$. Upon close examination of the proof however, for example in \cite{ST}, the only place where complete bounded curvature flows are assumed is in the application of Perelman's second pseudolocality theorem, and in view of its modification in Theorem \ref{pseudolocalityyilai}, we can extend this pseudolocality result to the incomplete flows $(V_k\times[0< T),g_k(t))$ (see Theorem \ref{asdasd}).  After appropriate scaling, this allows us to conclude  the following curvature bounds for $g(t)$ from  the condition \eqref{t1c1}:
  \[
  |\operatorname{Rm}|(x,t) \leq \frac{A(x)}{t}
  \]
  where the function $A(x)$ grows at some controlled rate (relative to $g_0$) on $M^3$ in terms of the function $f$. The Shrinking Balls Lemma {\cite[Corollary 3.3]{ST}} is then used to control distances relative to $g(t)$ in terms of distances relative to $g_0$ and the function $A(x)$, which by our choice of $f$ and hence $A(x)$, implies the completeness of $g(t)$.\\

Producing noncompact Ricci flows using compact approximations obtained through doubling an exhaustion was first done in \cite{CRW} in the case $(M^n, g_0)$ is complete with nonnegative complex sectional curvature.  Using splitting theorems on the universal cover, their proof reduced to the case when the soul is a single point and hence $M^n=	\mathbb{R}^n$.  In this case, using an exhaustion via the convex sublevel sets of the Busemann function, they were able to costruct the approximating compact manifolds without boundary $\{(N_k, h_k)\}$ to have positive complex sectional curvature and showed the corresponding Ricci flows $h_k(t)$ exist up to a uniform time $T>0$ and converge locally uniformly to a smooth, complete, nonnegative complex sectional curvature solution $g(t)$ on $M^n\times[0, T)$.  

 We observe that the above construction in \cite{CRW} can be combined with the above construction from \cite{yL} to give the following

   \begin{theorem}\label{t2}

   	Let $(M^3, g_0)$ be a 3-dimensional Riemannian manifold with nonnegative Ricci curvature, where $g_0$ is a compactly supported perturbation of a complete nonnegative sectional curvature metric.  Then the Ricci flow \eqref{RF} admits a smooth short-time solution $g(t)$ that starts from $g_0$, and is complete and has nonnegative Ricci curvature for every $t > 0$.
   	
   \end{theorem}

   % \begin{remark}
   %The problem with starting on a  general manifold $M^3$ in Theorem \ref{t2} is that the pullback of $g_0$ to the universal cover will no longer be a compactly supported perturbation of a nonnegative sectional curvature metric which poses a problem in our proof.
   %\end{remark}

\section*{Acknowledgement}The authors would like to thank Peter Topping for helpful comments.

\section{PRELIMINARIES}
 We begin with the following main existence result from \cite{yL}.

\begin{theorem}\label{tyl} [Convergence of 3D singular Ricci flows \cite{yL}] Let $(M^3, g_0)$ be a complete 3-dimensional Riemannian manifold with nonnegative Ricci curvature, $\{V_k\}$ an exhaustion of $M^3$ by relatively compact connected open sets, and $(N_k , h_k)$ a sequence of compact Riemannian manifolds without boundaries with diffeomorphisms onto their images $\phi_k: V_k \to N_k$ satisfying 	
\[
\phi_k^*(h_k)\xrightarrow[C^{\infty}_{loc}(M^3)]{}  g_0 .
\]
	
	  Then for each $V_l$ there exists an incomplete solution $h_{k_l}(t)$ to \eqref{RF} on $\phi_{k_l}(V_{l})\times[0, T)$ for some $k_l\geq l$ with $h_{k_l}(0)=h_{k_l}$ and $T$ independent of $l$ such that
	
	\begin{enumerate}
		\item Each $\{\phi_{k_l}(V_l)\times[0, T), h_{k_l}(t)]\}$ is embedded within a singular Ricci flow $\mathcal{N}_{k_l}$ emerging from $(N_{k_l}, h_{k_l})$.
		\item We have the convergence
		\[
		\phi_{k_l}^*(h_{k_l} (t))\xrightarrow[C^{\infty}_{loc}(M^3\times[0, T))]{}  g(t) 		
		\]
		 where $g(t)$ is a nonnegative Ricci curvature (but possibly incomplete for all $t>0$) solution to \eqref{RF} on  $M^3\times[0,T)$ with $g(0)=g_0$.
		
		\item If $(M^3, g_0)$ is the Riemannian double cover of $(M', g'_0)$, then $g(t)$ pushes down to a solution $g'(t)$ to \eqref{RF} on $M'\times[0, T)$.
	\end{enumerate} 
	\end{theorem}
We refer to \cite{yL} and \cite{KL} for the definition and properties of 3D singular Ricci flows.

Pseudolocality results for 3D Ricci flows were established by Hochard {\cite[Theorem 2.4]{raH}} and extended by Simon and Topping {\cite[Theorem 1.1]{ST}}.  The key feature in these results were that they applied to Ricci flows starting from arbitrary initial domains, whereas previous results required a sufficiently Euclidean initial domain.  This feature will be crucial for our application and proof of Theorem \ref{t1}.  The results from \cite{raH} and \cite{ST} were stated for local Ricci flows contained in some complete bounded curvature Ricci flow, and so do not directly apply to our setting.  However, what was actually proved in \cite{ST} for example, was that the results hold for any Ricci flow $(N\times[0, T),  g(t))$ for which {\cite[Theorem 6.2]{ST}} can be assumed (after removing the complete bounded curvature assumption there).  This assumption is possible in our setting due to Lai's extension of Perelman's pseudolocalicty 
to 3D singular Ricci flows \cite[Theorem 6.1]{yL}.  We may thus conclude

\begin{theorem}\label{asdasd}[Extension of {\cite[Theorem 1.1]{ST}} to 3D singular Ricci flows]
	
	Let $\{N \times [0, T), g(t)\}$ be a smooth (possibly incomplete) solution to \eqref{RF} embedded within some 3D singular Ricci flow $\mathcal{M}$, and let $p\in N$.  Suppose  that
	$ B_{g(0)}(p, 1+\sigma) \subset \subset N$ for some $\sigma>0$, that
	\begin{equation}
		\operatorname{Vol} B_{g(0)}(p, 1) \geq v_0 > 0 
	\end{equation}
	and
		\begin{equation}
		\operatorname{Ric}{(g(0))} \geq -K < 0 \text{ on } B_{g(0)}(p, 1 + \sigma).
		\end{equation}
 Then there exist $\widetilde{T} = \widetilde{T}(v_0, K, \sigma) > 0$, $\tilde{v}_0 = \tilde{v}_0(v_0, K, \sigma) > 0$, $\widetilde{K} = \widetilde{K}(v_0, K, \sigma) > 0$ and $c_0 = c_0(v_0, K, \sigma) < \infty$ such that for all $t \in [0, T) \cap (0, \widetilde{T})$ we have
 $ B_{g(t)}(p, 1) \subset \subset N$, and on $B_{g(t)}(p, 1)$ we have
 
	\begin{enumerate}
		\item $\operatorname{Vol} B_{g(t)}(p, 1) \geq \tilde{v}_0 > 0$,
		\item $\operatorname{Ric}(g(t)) \geq -\tilde{K}$,
		\item $|\operatorname{Rm}|_{g(t)} \leq \frac{c_0}{t}$.
	\end{enumerate}
\end{theorem}
 \begin{remark} The Theorem coincides with {\cite[Theorem 1.1]{ST}} when $(N\times[0, T), g(t))$ is contained in some complete bounded curvature Ricci flow.  In this case, conclusion (3) was established independently in {\cite[Theorem 2.4]{raH}}.  We will actually only need conclusion (3) for our later purposes.
 \end{remark}

\begin{proof}[Proof of Theorem \ref{asdasd}]

	We begin with the following statement of \cite[Theorem 6.1]{yL} which extends Perelman's first pseudolocality Theorem to 3D singular Ricci flows.

	\begin{theorem}\label{pseudolocalityyilai0}[Extension of \cite[Theorem 10.1]{P} to 3D singular Ricci flows]
		
		For every $\alpha > 0$, there exists $\delta, \epsilon > 0$ with the following property.
		
		Let $\{N \times [0, T), g(t)\}$ be a smooth (possibly incomplete) solution to \eqref{RF} embedded within some 3D singular Ricci flow $\mathcal{M}$ and let $p\in N$.  Suppose $B_{g(0)}(p, r_0) \subset \subset N$ and
		\begin{enumerate}
			\item $R(g(0)) \geq -r_0^{-2}$ on $B_{g(0)}(p, r_0)$,\,\,  ($R(g(0))$ denotes scalar curvature of $g(0)$),
			\item  $\operatorname{Vol}(\partial\Omega)^3 \geq (1 - \delta)c_3\operatorname{Vol}(\Omega)^2$
			for all $\Omega \subset B_{g(0)}(p, r_0)$ where $c_3$ is the Euclidean isoperimetric constant at dimension 3.
		\end{enumerate}
		Then $B_{g(t)}(p, \epsilon r_0)\subset \subset N$  and $$|\operatorname{Rm}|(x, t) < \alpha t^{-1} + (\epsilon r_0)^{-2}$$ holds on $B_{g(t)}(p, \epsilon r_0)$ for all $t\in [0, \text{min}(T, (\epsilon r_0)^2)]$.
	\end{theorem}
	
	We use this now to extend {\cite[Theorem 6.2]{ST}} (a modification of Perelman's second pseudolocality Theorem \cite[Theorem 10.3]{P}) to 3D singular Ricci flows.

	\begin{theorem}\label{pseudolocalityyilai}[Extension of {\cite[Theorem 6.2]{ST}} to 3D singular Ricci flows]

		Given $v_0>0$, there exists $\epsilon>0$ with the following property:
		Let $\{N \times [0, T), g(t)\}$ be a smooth (possibly incomplete) solution to \eqref{RF} embedded within some 3D singular Ricci flow $\mathcal{M}$ and let $p\in N$.  Suppose $B_{g(0)}(p, r_0) \subset \subset N$ and
		\begin{enumerate}
			\item $|\operatorname{Rm}|_{g(0)} \leq r_0^{-2}$ on $B_{g(0)}(p, r_0)$,
			\item  $\operatorname{Vol} B_{g(0)}(p, r_0) \geq v_0 r_0^3$.
		\end{enumerate}
		Then $B_{g(t)}(p, \epsilon r_0)\subset \subset N$  and $$|\operatorname{Rm}|(x, t) < (\epsilon r_0)^{-2}$$ holds on $B_{g(t)}(p, \epsilon r_0)$ for all $t\in [0, \text{min}(\epsilon r_0, T))$.
	\end{theorem}
	
	\begin{proof}[Proof of Theorem \ref{pseudolocalityyilai}]
		By scaling, we may assume $r_0 =1$.  By the results in \cite{CGT}, conditions (1) and (2) imply a uniform lower bound on the injectivity radius at $p$ depending only on $v_0$.  From this and the bound on curvature (actually we only need a lower bound on Ricci curvature), we may find harmonic coordinates $\mathbf{x}$ around $p$ in which we have $c(\|\mathbf{x}\|)^{-1}\delta_{ij} \leq g_{ij}(\mathbf{x}) \leq c(\|\mathbf{x}\|) \delta_{ij}$ for all $\|\mathbf{x} \| \leq d$ for some $d>0$, where the function $c(\rho)$ depends only on $v_0$, and $c\to 1$ as $\rho \to 0$.   In particular, if $\delta', \epsilon'$ correspond to $\alpha=1$ in Theorem \ref{pseudolocalityyilai0}, then conditions (1) and (2) in that theorem will hold on $B_{g(0)}(p, r)$ for some $r$ depending on $v_0$ and $\delta'$, and we conclude that $B_{g(t)}(p, \epsilon' r)\subset \subset N$  and $$|\operatorname{Rm}|(x, t) \leq t^{-1} + \epsilon'^{-2}$$ holds on $B_{g(t)}(p, \epsilon' r)$ for all $t\in [0, \min(T, \epsilon'^2 r)]$.  In particular, since $r\leq 1$ we have $|\operatorname{Rm}|(x, t) \leq 2t^{-1}$  on $B_{g(t)}(p, \epsilon' r)$ for all $t\in [0, \min(\epsilon'^2 r, T)]$ and it follows from {\cite[Theorem 3.1]{blC}} that we may have 
		$$|\operatorname{Rm}|(x) < C (\epsilon' r)^{-2}$$ 
		for a universal constant $C$ (depending only on dimension), and on $B_{g(t)}(p, \frac{\epsilon' r}{2})$ for all $t\in [0, \min(\frac{\epsilon' r}{2}, T)]$.  This completes the proof of the Theorem \ref{pseudolocalityyilai}.

	\end{proof}
		
Theorem \ref{asdasd} now follows, as described above,  by combining Theorem \ref{pseudolocalityyilai} with the proof of \cite[Theorem 1.1]{ST}.
	
\end{proof}

 Finally, we will make use of the following result which holds for any  solution (possibly incomplete) to \eqref{RF} in all dimensions.

\begin{proposition}\label{SBL}[Shrinking Ball Corollary 3.3 in \cite {ST}]
	There exists a dimensional constant $\beta = \beta(n) \geq 1$ such that the following holds.  
	
	Suppose $(M^n\times[0, T], g(t))$ is a (possibly incomplete) Ricci flow on an n-dimensional manifold $M^n$ such that $B_{g(0)}(x_0, r) \subset\subset M^n$ for some $x_0\in M^n$ and $r > 0$, and $Ric(g(t)) \leq (n-1)c_0/t$ on $B_{g(0)}(x_0, r) \cap B_{g(t)}(x_0, r-\beta\sqrt{c_0t})$ for each $t \in (0, T]$ and some $c_0 > 0$. Then $B_{g(0)}(x_0, r) \supset B_{g(t)}(x_0, r-\beta\sqrt{c_0t})$ for all $t \in [0, T]$.
\end{proposition}

\section{proof of Theorem \ref{t1}}
The proof is based on Proposition \ref{SBL} and the corollary to  Theorem \ref{asdasd} below.  Let us use $\widetilde{T}(v)$ and $c(v)$ to denote the positive functions $\widetilde{T}(v, -1, 1)$ and $c_0(v, -1, 1)$ for $0\leq v \leq V(3, -1/2)$ from Theorem \ref{asdasd}, where $V(3, -1/2)$ denotes the volume of the unit ball in the 3-dimensional space form of constant curvature $-1/2$.  Note that we may increase the given function $c_0(v, -1, 1)$ and decrease the function $\widetilde{T}(v, -1, 1)$ as we like without changing the statement of Theorem \ref{asdasd}, and so we may assume that $c(v)$ is strictly decreasing in $v$ and that 
\begin{equation}\label{tvsc} c(v)=1/\widetilde{T}(v)  \end{equation}
for all $v>0$.  We also note that by the example of the solution to Ricci flow emanating from arbitrarily sharp cones (see {\cite[Section 4 of Chapter 5]{CK}} for more detail), it must be that 
\[
\lim_{v\to 0} c(v)=\infty .
\]

\begin{corollary}\label{c1}[Corollary to  Theorem \ref{asdasd}]  
	
Let $\{N \times [0, T), g(t)\}$ for $T\leq 1$ be a smooth (possibly incomplete) solution to \eqref{RF} embedded within some 3D singular Ricci flow $\mathcal{M}$.  Suppose for some $x_0\in N$ and $v>0$ we have
	$ B_{g(0)}(x_0, 2\sqrt{c(v)}) \subset \subset N$ and
	\begin{equation}\label{c1e1}
		\frac{\operatorname{Vol} B_{g(0)}(x_0, \sqrt{c(v)} )}{(\sqrt{c(v)})^3} \geq v,
	\end{equation}
and 
\begin{equation}\label{c1e2}
\operatorname{Ric}{(g(0))} \geq  -1/c(v) \,\, \text{on}\,\, B_{g(0)}(x_0, 2\sqrt{c(v)}).
	\end{equation}
	Then for all $t\in (0,T)$  we have $ B_{g(t)}(x_0, \sqrt{c(v)}) \subset \subset N$ and 
\[
|\operatorname{Rm}|_{g(t)} \leq \frac{c(v)}{t} \text{ on } B_{g(t)}(x_0, \sqrt{c(v)}).
\]
\end{corollary}

\begin{proof}
		Let $g(t)\,\, \text{on}\,\, N\times[0, T)$ be as in the Theorem.  Write $\lambda=1/c(v)$ and consider the rescaled solution to \eqref{RF} given by 
	\begin{equation}
	g_{\lambda}(s):=\lambda g(s/\lambda)\,\, \text{on}\,\, N\times[0,\lambda T).
	\end{equation}
	 Then we have 
	\begin{equation}
		\operatorname{Ric}{(g_{\lambda}(0))} \geq  -1 \,\, \text{on}\,\, B_{g_{\lambda}(0)}(x_0, 2),
	\end{equation}
	
	and
	
	\begin{equation}
		\begin{split}
		\operatorname{Vol}_{g_{\lambda}(0)} B_{g_{\lambda}(0)}(x_0, 1)&=\frac{\operatorname{Vol}_{g_0} B_{g_{0}}(x_0, 1/\sqrt{\lambda})}{(1/\sqrt{\lambda})^3}\geq v\\
		\end{split}
	\end{equation}
where we have used the definition of $g_{\lambda}$  and \eqref{c1e1}.  Thus by conclusion (3) in Theorem \ref{asdasd} and \eqref{tvsc} we have 
\[
|\operatorname{Rm}|_{g_\lambda (s)} \leq \frac{c(v)}{s} \,\,\text{on}\,\, B_{g_\lambda (s)}(x_0, 1) \,\, \text{for} \,\, s\in [0, \lambda T =\tilde{T}(v) T ).
\]
This in turn gives 
\begin{equation}
|\operatorname{Rm}|_{g (t)}=\lambda |\operatorname{Rm}|_{g_{\lambda} (\lambda t)} \leq \lambda \frac{c(v)}{\lambda t}=\frac{c(v)}{ t} \,\,\text{on}\,\, B_{g(t)}(x_0, \sqrt{c(v)})\end{equation}
for $t\in (0, T)$.  This concludes the proof of the Corollary.
\end{proof}

We now finish the proof of Theorem \ref{t1}. Let $(M^3, g_0)$ be a complete 3-manifold with $\operatorname{Ric}(g_0)\geq 0$ that satisfies the volume decay assumption \eqref{t1c1} for the function 
\[
f(r):= 2c^{-1}(r^2),
\]
where $c^{-1}$ is the inverse of the function $c(v)$ discussed above. By the properties of $c(v)$, the function $f(r)$ is positive and defined on $[L, \infty)$ for some $L> 0$ and satisfies $\lim_{r\to \infty} f(r)=0$.  \\

 Assume first that $(M^3, g_0)$ is orientable.  Then we may find an exhaustion of $M^3$ by relatively compact connected sets $V_k$ with smooth boundaries, and a sequence of smooth compact Riemannian manifolds without boundaries $(N_k, h_k)$ and diffeomorphisms onto their images $\phi_k:V_k\to N_k$ converging to $(M^3, g_0)$ as in the hypothesis of Theorem \ref{tyl}.  We conclude by Theorem \ref{tyl} the existence of corresponding local solutions $(\phi_{k_l}(V_l)\times[0, T), h_{k_l}(t))$ converging to a possibly incomplete solution $(M^3\times[0, T), g(t))$ to \eqref{RF} with $g(0)=g_0$ and $\operatorname{Ric}(g(t))\geq 0$ for all $t\in [0,T)$.  We may also apply Corollary \ref{c1} to each local solution $h_{k_l}(t)$ by Theorem \ref{tyl} (1).  
 
 For simplicity, we will denote the 
sequence  $(\phi_{k_l}(V_l)\times[0, T), h_{k_l}(t))$ by $(W_l\times[0, T), H_l(t))$.  Let $T'=\min (T, 1, \beta^{-1})$ where $\beta=\beta(3)>0$ is the constant from Proposition \ref{SBL}.
The assumed volume bound \eqref{t1c1} and the definition of $f$ ensure that $(M^3, g_0)$ satisfies: 
\[
\frac{\operatorname{Vol} B_{g(0)}(x_0, \sqrt{c(v)} )}{(\sqrt{c(v)})^3} \geq 2v
\]
for all sufficiently small $v>0$. The local smooth convergence of the $(W_l\times[0, T'), H_l(t))$'s to $(M^3\times[0, T'), g(t))$ implies: for each $v>0$ there exists $m_v$ such that for each $l\geq m_v$, the incomplete solution  $(W_l\times[0, T'), H_l(t))$ satisfies the hypothesis of Corollary \ref{c1} with $x_0=p_{l}:=\phi_{k_l}(p)$ and hence
\[
|\operatorname{Rm}|_{H_l(t)} \leq \frac{c(v)}{t} \text{ on } B_{H_l(t)}(p_{l}, \sqrt{c(v)}) \,\, \text{ for all } t \in [0, T').
\]
 Thus by Proposition \ref{SBL} we conclude that for each $v>0$ sufficiently small  we have  
\[
B_{H_l(t)}(p_{l}, \sqrt{c(v)}-\beta \sqrt{c(v)t} )\subset B_{H_l(0)}(p_{l}, \sqrt{c(v)})
\]
 in $W_l$ for all $l\geq m_v$ and $t\in [0, T')$. Thus in the limit we have
$$B_{g(t)}(p, \sqrt{c(v)}-\beta \sqrt{c(v)t} )\subset B_{g(0)}(p, \sqrt{c(v)})$$ in $M^3$ for all for all $v>0$ sufficiently small and $t\in [0, T')$.  It follows by the completeness of $g_0$ and the fact that $\lim_{v\to 0}c(v)=\infty$, that $(M^3, g(t))$ is complete for all $t< T'$.   This concludes the proof of Theorem \ref{t1} assuming $(M^3, g_0)$ is orientable.  If it is not orientable, we repeat the argument to obtain a solution $\tilde{g}(t)$ to \eqref{RF} on the Riemannian double cover $(\tilde{M}, \tilde{g}_0)$, then by Theorem \ref{tyl} (3) we can push this down to a solution $g(t)$ on $M^3$ having the desired properties.  

 This completes the proof of Theorem \ref{t1}.

\section{proof of Theorem \ref{t2}}

By Liu's classification \cite{gL}, either $M^3=\mathbb{R}^3$ or else the Riemannian universal cover of $(M^3, g_0)$ is a Riemannian product.  By the results for Ricci flow on surfaces by Topping \cite{TT} and Giesen, Topping \cite{GT}, we may thus assume that $M^3=\mathbb{R}^3$.

Let $(M^n, g)$ be a complete Riemannian manifold with nonnegative complex sectional curvature of arbitrary dimension $n$.  It was proved in \cite{CRW} that when the soul is a single point, so that $M^n=\mathbb{R}^n$, there exists an exhausting sequence of relatively compact connected open sets $\{V_k\subset M^n\}$, a double sequence of compact Riemannian manifolds without boundaries $\{(N_k, h_{k, l})\}$  and diffeomorphisms onto their images $$\phi_k: V_k \to N_k;\,\,\,\,\,\,\,\psi_k: N_k \to N_k $$  all together satisfying the following:
	\begin{enumerate}
	\item Each $(N_k, h_{k, l})$ has strictly positive complex sectional curvature, volume uniformly bounded from below, and diameter bounded above depending on $k$ but not on $l$.

	\item Each $\psi_k$ is an isometry relative to every $h_{k, l}$ and satisfies $$\psi_k^2 =\operatorname{Id}\neq \psi_k;\,\,\, \,\,\,\,\,\,\, \psi_k (q)=q \,\,\text{iff}\,\, q\in  \partial (\phi_k V_k )$$ 
	
	\item  For all $q\in V_k$ we have $$dist_g (q, \partial V_k) \geq dist_{\phi^*(h_{kl})} (q, \partial V_k)) - C$$ for some $C$ independent of $k, l$. 
	
	\item Given any compact set $S\subset \subset M^n$, there exists $k_0$ such that for every $k > k_0$ we have the smooth convergence
		\[
		\phi_k^*(h_{k, l})\xrightarrow[l\to \infty]{} g \,\,\, \text{on} \,\, S	
		\]
 where the convergence is uniform over $k$.
\end{enumerate}	
It was then proved that the corresponding Ricci flows $h_{k, l}(t)$ exist on $N_k$ up to a uniform time $T>0$ independent of $k$ and $l$, and that a diagonal subsequence of $\phi_k^* (h_{k, l} (t))$ converges smoothly uniformly on compact subsets on $M^n\times[0, T)$ to a solution $g(t)$ to Ricci flow which has nonnegative complex sectional curvature and is complete for all $t\in[0, T)$.

 Now assume that $n=3$ above, in which case $g$ will in fact have nonnegative sectional curvature on $M^3=\mathbb{R}^3$.  Let $\tilde{g}$ be a compactly supported symmetric 2-tensor on $M^3$ such that $g_0:=g+\tilde{g}$ is a complete Riemannian metric with nonnegative Ricci curvature.  In other words, there is a compact set $K\subset \subset M^3$ for which 
\begin{equation}\label{ee1}
\operatorname{Ric}(g+\tilde{g})\geq 0 \text{ on }M^3, \text{ and } \tilde g =0  \text{ on } M^3\setminus K.
\end{equation}

 From now on, consider $k$ sufficiently large so that $K \subset \subset V_k$.   Define the smooth metrics $\tilde{h}_{k, l}$ on $N_k$ as

 \begin{equation}
 	\left\{
 	\begin{aligned}
 		\tilde{h}_{k, l} &:=h_{k, l} +(\phi_k^{-1})^*\tilde{g};\,\,\, \text{on}\,\,\, \phi_k(V_k) \\
 		\tilde{h}_{k, l} &:= \psi_k^* (h_{k, l} +(\phi_k^{-1})^*\tilde{g}) \,\,\, \text{on}\,\,\, \psi_k (\phi_k V_k).
 	\end{aligned}
 	\right.
 \end{equation}

% $$\tilde{h}_{k, l}(q):=h_{k, l}(q) +(\phi_k^{-1})^*\tilde{g}(q)\,\, \text{for} \,\, q \in \phi_k (V_k)$$
  
    Though the $\tilde h_{k,l}$ may not be positive definite a priori, we may assume that they are by taking $k,l$ sufficiently large and using property (4) as well as the fact that
\[
\inf\{\|v\|_{g+\tilde g} \; : \; v\in T_p M^3 \text{ where } p\in K \text{ and } \|v\|_{g}=1 \}>0, 
\]
where the positivity is to due the compactness of $K$.

 Then by property (1) above, each $(N_k, \tilde{h}_{k, l} )$ will satisfy:
\begin{equation}\label{eee1}\text{Vol}_{k, l} \geq v; \,\,\, \text{Diam}_{k, l} \leq C_k;\,\,\, \text{Ric}(\tilde{h}_{k, l}) \geq -c_{l}\end{equation}
for positive constants $v, C_k, c_l$ depending only on their subscripts (if any) and where $c_l \to 0$ as $l\to \infty$. 
It follows from Corollary 3 in \cite{BCRW} that for each $l$ sufficiently large depending on $k$, there exists a nonnegative Ricci curvature metric on $N_k$, and that these can be taken to converge on $N_k$ as $l\to \infty$ as we now describe.  In particular, the proof there showed that for each fixed $k$ and all sufficiently large $l$ the Ricci flow $\tilde{h}_{k, l}(t)$ starting from $\tilde{h}_{k, l}$ on $N_k$ exists up to a uniform time $T_k>0$ depending only on $k$ and  (by Theorem 2 in \cite{BCRW}) satisfies the curvature bound 

\begin{equation}\label{eeee1} 
	|\operatorname{Rm}_{k, l}(t)| \leq C'/t \end{equation}
 for some $C' >0$ independent of $k, l$.  Moreover, it was shown that the solutions $\{(N_k\times[0, T_k), \tilde{h}_{k, l}(t))\}_{l\in \mathbb{N}}$ subconverges, as in Hamilton's Compactness Theorem \cite{rH4}, to a limit solution $(N_k\times(0, T_k), \tilde{h}_{k, \infty}(t))$ having everywhere nonnegative Ricci curvature.  The nonnegativity of Ricci curvature can also be seen for example, from the bounds \eqref{eee1}, \eqref{eeee1} and {\cite[Lemma 2.2]{ST}}) which in particular imply a uniform lower bound on \begin{equation}\label{PP11}Rc(h_{kl}(t))\geq -100 c_l C' \end{equation} 
 for $t\in [0, T_k)$ provided $T_k$ is sufficiently small depending only on $C'$.
 
 Now the estimate \eqref{eeee1} and {\cite[Theorem 3.1]{blC}} imply that $ \tilde{h}_{k, \infty}(t)$ converges smoothly as $t\to 0$ on a given compact set $S\subset \subset N_k$ provided $\tilde{h}_{k, l}(0)$ likewise converges as $l\to \infty$.   On the other hand, for a given compact set and $k$ sufficiently large, the latter limit exists and equals $(\phi_k^{-1})^* (g + \tilde{g})$ by condition (4) above.  Moreover, by \eqref{PP11} and the fact $c_l \to 0$, we may still have condition (3) after replacing $h_{k, l}$ there with $\tilde{h}_{k, \infty}(t)$  where the constant $C$ there will be independent of $k, l$ and $t\leq \min(T_k, 1)$.
 
  In summary, we conclude the existence of sequences $l_k \to \infty$ and $t_k\to 0$ for which the metrics $H_{k}:= \tilde{h}_{k, l_k}(t_k)$ on $N_{k}$ satisfy the following relative to the same maps $\phi_k: V_k\subset M^3\to N_k$ and $\psi_k:N_k\to N_k$ defined above:
 
 \begin{enumerate}
 	\item [(a)] Each $(N_{k}, H_k)$ has nonnegative Ricci curvature.
 	\item [(b)] Each $\psi_{k}$ is an isometry relative to $H_k$ satisfying $$\psi_{k}^2 =\operatorname{Id}\neq \psi_k	\,\,\, \,\,\,\,\,\,\, \psi_k (q)=q \,\,\text{iff}\,\, q\in \partial \phi_k (V_k )$$ 
 	
 	\item  [(c)] For all $q\in V_k$ we have $$dist_g (q, \partial V_k) \geq dist_{\phi^*(H_k)} (q, \partial V_k)) - C$$ for some $C$ independent of $k$. 
 	\item [(d)] Given any compact set $S\subset \subset M^3$, we have the smooth convergence as $k\to \infty$
 	
 	\begin{equation}\phi_{k}^*(H_k)\xrightarrow[C^{\infty}_{loc}(M^3)]{}  (g + \tilde{g})
 		\,\,\, \text{on} \,\, S.	\end{equation}
 \end{enumerate}

 Now let $H_k(t)$ be the corresponding Ricci flow on $N_k$ with $H_k(0)=H_k$.  By Hamilton's convergence results for nonnegative Ricci curvature metrics in \cite{rH1}} we know that $H_k(t)$ is either stationary/Ricci flat or else exists up to some $0<T_k <\infty$ with $\text{Vol}_{H_k (t)} N_{k} \to 0$ as $t \to T_k$.  On the other hand, Perelman's pseudolocality \cite{P} combined with condition (c) above implies that for any given compact $S\subset \subset M^3$, there exists $T_S , V_S>0$ such that 
 $\text{Vol}_{H_k (t)} \phi_{k}(S) > V_S$ for all $t\leq \text{min}(T_S, T_k)$ and all $k$.  We conclude that $T_k >T>0$ for all $k$ and some $T>0$. 
 
 Thus from (a)-(c) and Theorem \ref{tyl} we obtain a nonnegative Ricci curvature (albeit possibly incomplete) solution $(M^3\times[0, T), g(t))$ to \eqref{RF} starting from $g(0)=g+\tilde{g}$ and after possibly shrinking $T>0$.  Moreover, from the proof of Theorem \ref{tyl} in \cite{yL}, we may actually conclude that  $(M^3\times[0, T), g(t))$ is a local limit of the solutions $\{N_{k}\times[0, T), H_k(t)\}$ to \eqref{RF} as in Theorem \ref{tyl} part (2).

 It remains to prove that $g(t)$ is complete for all $t>0$.  This in fact follows from the proof of completeness of the limit solution in \cite{CRW} as we now sketch.
 
The proof is based on the choice of the exhaustion $\{V_k\}_{k=1}^{\infty}$ of $M$ made in \cite{CRW}.   Specifically, $V_k$ was defined as the $k$ sublevel of the Busemann function based at some $p_0\in M$.   In particular, if $\mathcal{B}$ denotes the set of geodesic rays on $M$ starting from $p_0$ then  $$V_k:=\{q\in M: b(q)<k\}$$ where
 \begin{equation}\nonumber
 	b(q):= \sup_{\gamma \in \mathcal{B}}  \lim_{t\to \infty}\left( t- \text{dist}_g(\gamma(t), q) \right).
 \end{equation}
 
 In what follows, we fix some $T'<T$ and some $R>0$. We will use $C$ to denote a positive constant depending only on the solution $g(t)$ on $M\times[0, T']$ and which may differ from line to line.  For each $k$, denote $$p_k'=\phi_k (p_0) \in N_k;\,\,\,  V_k'=\phi_k (V_k) \subset N_k;\,\,\, L_k=\text{dist}_{H_k(0)}(p_k' ,  \partial V_k').$$
In particular, we have that $dist_{H_k(t)}(p_k', \psi_k (p_k'))=2L_k$ by property (b) of the map $\psi_k$.

By the local convergence of the $\phi_k^* (H_k(t)) \to g(t)$ on $M$, and the fact that $\psi_k$ is an isometry relative to $H_k$, there are neighborhoods $U, V$ around $p_k', \psi_k (p_k')$ (resp.) such that \begin{equation}\label{AM1}\sup_{(U\bigcup V) \times [0, T']}|Rc(H_k(t))|\leq C,\end{equation} for some $k$ sufficiently large.  From \eqref{AM1}, the argument in  \cite{CRW} using the Ricci flow \eqref{RF} and the second variation formula for arc length (see also {\cite[Theorem 17.4]{rH4}}), we have that 
\begin{equation}\label{CRWe0}
	dist_{H_k(t)}(p_k', \psi_k (p_k')) \geq 2 L_k -C.\end{equation} 
The LHS above equals $2dist_{H_k(t)}(p_k', \partial V_k)$.  Moreover, for $k$ sufficiently large we have $B_{H_k(t)}(p_0,R)\subset  V_k$.  Combining these with  \eqref{CRWe0}  gives \begin{equation}\label{CRWpe11}dist_{H_k(t)}(B_{H_k(t)}(p_k',R), \partial V_k') \geq L_k -C -R.\end{equation}   
for $k$ sufficiently large and all $t\in [0, T']$.  On the other hand, we have $H_k(0)\geq H_k(t)$ by \eqref{RF} and the fact $Rc(H_k(s)) \geq 0$ for all $s\in [0, t]$. Thus by property (c) above for the metric $H_k$,  we may replace $\text{dist}_{H_k(t)}$ in \eqref{CRWpe11} with $\text{dist}_{(\phi^{-1})^*g}$ provided we subtract a constant $C$ from the RHS.  Next, by the smooth local convergence of the $\phi_k^* (H_k(t))'s$ to $g(t)$ on $M$, we may further replace $B_{H_k(t)}(p_k',R)$ with  $B_{g(t)}(p_0,R)$ by futher subtracting a constant $C$ from the RHS.  Pulling the resulting inequality back to $V_k$ by $\phi^*$ gives
%\begin{equation}\label{CRWpe1}dist_{g}(p_0, \partial V_k) \geq L_k -Ct -C.\end{equation} 
%for all $t\in [0, T']$ and some $C$. In particular, for any $R>0$ and sufficiently large $k$ so that $B_g(p_0,R)\subset  V_k$ we have

\begin{equation}\nonumber dist_{g}(B_{g(t)}(p_0,R), \partial V_k) \geq L_k -C -R,\end{equation}    
for $k$ sufficiently large and all $t\in [0, T']$. Now we note the following basic propery of the sublevel sets of the Busemann function: for any $s_1 < s_2$ we have
\begin{equation}\label{AM2}b^{-1}((-\infty, s_1] )= \{ q\in b^{-1}((-\infty, s_2]): \text{dist}_g(q, \partial b^{-1}((-\infty, s_2]) \geq s_2-s_1.\}\end{equation}
Combining these with the fact that $\partial V_k=\partial b^{-1}((-\infty, k])$ gives
\begin{equation}\nonumber
	B_{g(t)}(p_0,R) \subset b^{-1}((-\infty, k- (L_k -C-R)]) \subset b^{-1}((-\infty, C+R])
\end{equation}
for all $t\in [0, T']$ where for the last inclusion we have used that $b(p_0)=0$ and thus $L_k\geq k$ again by the above property of $b$.  

 In particular, we have shown that for all $R>0$ there is some compact set $K_R \subset \subset M$ such that $B_{g(t)}(p_0,R) \subset K_R$ for all $t\in [0, T']$ and it follows that $g(t)$ is complete on $M$ for each $t\in [0, T']$ and thus for each $t\in [0, T)$ as $T'<T$ was arbitrarily chosen.  This completes the proof of Theorem \ref{t2}.

\end{document}